\documentclass{article}

\usepackage{subfigure}
\usepackage{tikz}
\usetikzlibrary{graphs,graphs.standard,calc}
\usepackage[utf8]{inputenc}
\usepackage[utf8]{inputenc}
\usepackage[english]{babel}
\usepackage{amsmath}
\usepackage{times}
\usepackage{graphicx}
\usepackage{float}
\usepackage[margin=1in]{geometry}
\usepackage{blindtext}
\usepackage{amssymb}
\usepackage{tikz}
\usepackage{amsthm}
\usepackage{mathtools}
\usepackage{textcmds}

\usepackage{comment}
\usepackage[mathlines]{lineno}

\newtheorem{theorem}{Theorem}[section]

\newtheorem{lemma}[theorem]{Lemma}

\theoremstyle{remark}
\newtheorem{remark}{Remark}
\newtheorem{question}{Question}
\newtheorem{case}{Case}
\newtheorem{claim}{Claim}[theorem]
\theoremstyle{definition}

\newcommand{\bbL}{\mathbb{L}}
\newcommand{\bx}{\mathbf{x}}
\newcommand{\rx}{\mathrm{x}}

\newcommand{\rex}{\mathrm{ex}}
\newcommand{\rEX}{\mathrm{EX}}
\newcommand{\rspex}{\mathrm{spex}}
\newcommand{\rSPEX}{\mathrm{SPEX}}

\title{Spectral Tur\'an problems for intersecting even cycles}
\date{\today}
\author{Dheer Noal Desai
\thanks{Department of Mathematics and Statistics, University of Wyoming, \texttt{ddesai1@uwyo.edu}}
\thanks{Department of Mathematical Sciences, The University of Memphis, \texttt{dnddesai1@memphis.edu}}}

\begin{document}

\maketitle
\begin{abstract}
Let $C_{2k_1, 2k_2, \ldots, 2k_t}$ denote the graph obtained by intersecting $t$ distinct even cycles $C_{2k_1}, C_{2k_2}, \ldots, C_{2k_t}$ at a unique vertex. In this paper, we determine the unique graph with maximum adjacency spectral radius among all graphs on $n$ vertices that do not contain any $C_{2k_1, 2k_2, \ldots, 2k_t}$ as a subgraph, for $n$ sufficiently large. 
When one of the constituent even cycles is a $C_4$, our results improve upper bounds on the Tur\'an numbers for intersecting even cycles that follow from more general results of F\"{u}redi \cite{furedi1991turan} and Alon, Krivelevich and Sudakov \cite{alon2003turan}.
Our results may be seen as extensions of previous results for spectral Tur\' an problems on forbidden even cycles $C_{2k}, k\ge 2$ (see \cite{cioabua2022evencycle, Nikiforovpaths, zhai2020spectralhexagon, ZW}).  
\end{abstract}

\section{Introduction}
For any graph $F$, the Tur\' an number $\rex(n, F)$ is the maximum number of edges in any graph on $n$ vertices that avoids any isomorphic copy of $F$ as subgraphs. The set of $n$-vertex $F$-free graphs with $\rex(n, F)$ many edges is denoted by $\rEX(n, F)$ and called the set of \textit{extremal graphs}. For the complete graph, $K_{r+1}$, Tur\' an \cite{Turan41} proved that the Tur\' an graph $T_r(n)$ is the only extremal graph, where the Tur\' an graph is the unique complete $r$-partite graph on $n$ vertices with each part having either $\left\lfloor \frac{n}{r} \right\rfloor$ or $\left\lceil \frac{n}{r} \right\rceil$ vertices.

The celebrated Erd\H{o}s-Stone-Simonovits theorem \cite{erdos1966limit,ES46} extends Tur\' an's theorem to other graphs with chromatic number $r+1$ and gives the exact asymptotics for the Tur\' an numbers of graphs with chromatic number at least three. It states that if the chromatic number of $F$ is $r+1$, then \[\rex(n, F) = \left(1 - \dfrac{1}{r} + o(1)\right)\dfrac{n^2}{2}.\]
Consequently, when $F$ is a bipartite graph and $r=1$, we only get $\rex(n, F) = o(n^2)$ and it remains to determine the exact asymptotics for several basic bipartite graphs. For example, for bipartite graphs $K_{s, t}$ with $s \le t$, the K\H{o}vari-S\' os-Tur\' an theorem \cite{kovari1954problem} establishes the upper bound $\rex(n, K_{s, t}) = O(n^{2-1/s})$, however matching lower bounds have only been confirmed for complete bipartite graphs $K_{s, t}$ with $t \ge (s-1)! + 1$ \cite{alon1999norm, furedi1996new}. 
Another avenue for  difficult Tur\' an problems comes from even cycles. The order of magnitude for $\rex(n, C_{2k})$ is determined only for $k=2, 3$ and $5$ \cite{furedi2013history}. 

Let $C_{k_1, k_2, \ldots, k_t}$ be the graph obtained by intersecting $t$ cycles of lengths $k_i \ge 3$, for $1 \le i \le t$, at a unique vertex. If every $k_i$ is odd, we call such a graph an intersecting odd cycle. In \cite{ET} Erd\H{o}s, F\" uredi, Gould and Gunderson determined the exact number $\rex(n, F)$ and set of extremal graphs $\rEX(n, F)$ for $F = C_{3, 3, \ldots, 3}$ the friendship graph and $n$ sufficiently large. More recently, Hou, Qiu, and Liu \cite{hou2015extremal, hou2018turan} and Yuan \cite{yuan2018extremal} solved the Tur\' an problem for all other intersecting odd cycles $C_{2k_1+1, 2k_2+1, \ldots, 2k_t+1}$. 

The Tur\' an problem remains wide open for intersecting even cycles $C_{2k_1, 2k_2, \ldots, 2k_t}$.  For simplicity, assume that $k_1 \le k_2 \le \ldots k_{t-1} \le k_t$. 

F\"{u}redi \cite{furedi1991turan} and Alon, Krivelevich and Sudakov \cite{alon2003turan} generalized the K\H{o}vari-S\' os-Tur\' an theorem to bipartite graphs $H$ where one of the parts has maximum degree at most $r$. They determined $\rex(n, H) = O(n^{2-1/r})$. This upper bound follows from
Theorem 2.2 of \cite{alon2003turan} and we record it below for reference.
\begin{theorem}
\cite{alon2003turan}
\label{alonturan}
Let $H = (A \cup B, F)$ be a bipartite graph with sides $A$ and $B$ of sizes $|A| = a$ and $|B| = b$, respectively. Suppose that the degrees of all vertices $b \in B$ in $H$ do not exceed $r$. Let $G=(V, E)$ be a graph on $n$ vertices with average degree $d = 2|E(G)|/n$. If
\[\frac{d^r}{n^{r-1}} - \binom{n}{r}\left(\frac{a+b-1}{n}\right)^r > a-1,\]
then $G$ contains a copy of $H$.
\end{theorem}

When $r=2$, this directly gives that $\mathrm{ex}(n, H) \le \dfrac{1}{2}\left(a-1 + \frac{(a+b-1)^2}{2} - O(1/n)\right)^{1/2}n^{3/2}$.

It follows from Theorem \ref{alonturan} that \begin{equation}
\label{AKS bounds}
\rex(n, C_{2k_1, 2k_2, \ldots, 2k_t}) \le \dfrac{1}{2}\left(\kappa + \frac{(2\kappa + t)^2}{2} - O(1/n)\right)^{1/2}n^{3/2},   
\end{equation}where $\kappa=\sum_{i=1}^t (k_i - 1)$.

When $r=2$, Conlon and Lee \cite{conlon2021extremal} further showed that $\rex(n, H) = cn^{3/2}$, where $c$ is some constant, only if $H$ contains a $C_4$.
Thereafter,
Conlon, Janzer and Lee \cite{conlon2021more} further strengthened Theorem \ref{alonturan} as follows.

\begin{theorem}
\cite{conlon2021more}
\label{bounds on turan number for max degree r no C4}
Let $H$ be a  bipartite graph such that in one of the parts all the degrees are at most $r$ and $H$ does not contain $C_4$ as a subgraph. Then  
$\mathrm{ex}(n, H) = o(n^{2 - 1/r}).$
\end{theorem}

More recently, Sudakov and Tomon \cite{sudakov2020turan} strengthened the result as follows. 

\begin{theorem}
\cite{sudakov2020turan}
\label{No C4 bounds on turan number for max degree r}
Let $t \ge 2$ be an integer. Let $H$ be a $K_{t, t}$-free  bipartite graph such that every vertex in one of the parts of $H$ has degree at most $t$. Then $\rex(n,H) = O(n^{2 - 1/t}).$
\end{theorem}

In fact, if we only consider intersecting even cycles, then an earlier work of Faudree and Simonovits (see Theorem 2 of \cite{faudree1983class}) indirectly gives $\rex(n, C_{2k_1, 2k_2, \ldots, 2k_t}) = O(n^{1+1/k_1})$. This strengthens  \eqref{AKS bounds} when $k_1 \ge 3$.  

In this paper, we study a spectral version of the Tur\' an problem for intersecting even cycles. Analogous to the Tur\' an problem, let $\rspex(n, F)$ denote the maximum spectral radius of the adjacency matrix of any $F$-free graph on $n$ vertices, and $\rSPEX(n, F)$ denote the set of \textit{spectral extremal graphs} on $n$ vertices with adjacency spectral radius equal to $\rspex(n, F)$ and having no isomorphic copies of $F$ as subgraphs. 
Nikiforov \cite{Nikiforovpaths} pioneered the systematic study of spectral Tur\' an problems, although several sporadic results appeared earlier. In \cite{nikiforov2007bounds} Nikiforov proved that the only spectral extremal graph for $K_{r+1}$ is the Tur\' an graph $T_r(n)$. When combined with the observation that the average degree of a graph lower bounds the spectral radius of its adjacency matrix, Nikiforov's result implies and strengthens Tur\' an's theorem for complete graphs. Additionally, Nikiforov \cite{nikiforov2010contribution}, Babai and Guiduli \cite{babai2009spectral} proved spectral versions of the K\H{o}vari-S\' os-Tur\' an theorem for complete bipartite graphs $K_{s, t}$. Moreover, using the average degree bound these match the best known upper bounds for $\rex(n, K_{s, t})$, obtained by F\" uredi \cite{furedi1996upper}.  
Recently, several papers have been published determining $\rspex(n, F)$ for various families of graphs (see \cite{cioabua2022spectralerdossos, cioabua2022spectral, lu2006new, nikiforov2008spectral, shi2007upper, wilf1986spectral, yuan2012spectral, ZW, zhai2020spectral}). In fact, our proof further develops techniques that appear in \cite{cioabua2022spectralerdossos, cioabua2022evencycle}.
Spectral Tur\' an problems fit into a broader framework of problems called {\em Brualdi-Solheid problems} \cite{BS} that investigate the maximum spectral radius among all graphs belonging to a specified family of graphs. Several results are known in this area (see \cite{BZ, BLL, EZ, FN, nosal1970eigenvalues, S, SAH}). 

As mentioned earlier, solving Tur\'an problems for even cycles has proven to be a difficult task. 
Nikiforov conjectured the spectral extremal graphs for even cycles in \cite{ Nikiforovpaths}. Recently, the conjectures were proved and $\rSPEX(n, C_{2k})$ were determined for all values of $k$ when $n$ is sufficiently large. The case $k=2$ was covered in \cite{Nikiforovpaths,ZW}, the case $k=3$ was covered in \cite{zhai2020spectralhexagon}, and all other values of $k$ were covered in \cite{cioabua2022evencycle}. 

In \cite{cioabua2020maximum} Cioab\u a, Feng, Tait and Zhang solved the spectral Tur\' an problem for friendship graphs $C_{3, 3, \ldots, 3}$. Further, Li and Peng \cite{li2022spectral} solved the spectral extremal graphs for all other intersecting odd cycles $C_{2k_1+1, 2k_2+1, \ldots, 2k_t+1}$.  In \cite{intersectingcliques}, the authors generalized the results of \cite{cioabua2020maximum} to intersecting cliques. 

Let $M_k := \left\lfloor\frac{k}{2}\right\rfloor K_2 \cup \gamma K_1$, where $\gamma = k \pmod{2}$, denote a maximal matching on $k$ vertices.   
Let $S_{n,k} := K_k \vee (n-k) K_1$, denote the join of a clique on $k$ vertices and an independent set on $n-k$ vertices,  and $S_{n, k}^+$ be the graph obtained by adding an edge to the independent set of $S_{n,k}.$ Finally, let $F_{n,k} := K_k \vee M_{n-k} = K_k \vee \left(\left\lfloor\dfrac{n-k}{2}\right\rfloor K_2 \cup \gamma K_1\right)$ with $\gamma = (n-k) \pmod{2}$, be the graph obtained from $S_{n,k}$ by adding a maximal matching in the independent set of $S_{n,k}$.
For an integer $k$, let $k':= k - 1$. We prove the following theorems for intersecting even cycles.

\begin{theorem}
\label{theorem intersecting even cycles}
Let $\max\{k_1, k_2, \ldots, k_t\} \ge 3$ and $\kappa := \sum_{i=1}^t k_i'$. If $G$ is a graph of sufficiently large order $n$, with
$\lambda(G) \ge \lambda(S_{n, \kappa}^+)$ then $G$ contains $C_{2k_1, 2k_2, \ldots, 2k_t}$ unless $G = S_{n, \kappa}^+$.

\end{theorem}

In Theorem \ref{theorem intersecting even cycles}, we are assuming that there is at least one even cycle with more than $6$ vertices. In case all the even cycles have four vertices, then we have the following theorem.

\begin{theorem}
\label{theorem intersecting 4-cycles}
      If $G$ is a graph of sufficiently large order $n$, with $\lambda(G) \ge \lambda(F_{n, t})$, then $G$ contains $C_{4,4, \ldots, 4}$ unless $G=F_{n, t}$, where $C_{4,4, \ldots, 4}$ denotes the intersecting even cycle consisting of $t$ cycles of length four that intersect at a unique vertex.
\end{theorem}

A consequence of Theorems \ref{theorem intersecting even cycles} and \ref{theorem intersecting 4-cycles} is that $\rex(n, C_{2k_1, 2k_2, \ldots, 2k_t}) \le \frac{1}{2}(\kappa + o(1))n^{3/2}$. This improves the bounds in  \eqref{AKS bounds} and is not covered by Theorem \ref{bounds on turan number for max degree r no C4} if $k_i = 2$ for some $i$, since one of the even cycles is a $4$-cycle.

Interestingly, very recently, Fang, Zhai and Lin \cite{longfei2023spectral} found similar spectral extremal graphs for disjoint even cycles and showed that $\rSPEX(n, tC_{2k}) = S_{n, lt-1}^+$ for $k \ge 3$, and  $\rSPEX(n, tC_{4}) = F_{n, 2t-1}$. We note that Theorems  \ref{theorem intersecting even cycles} and \ref{theorem intersecting 4-cycles} along with the results from \cite{longfei2023spectral} may be viewed as extensions of the spectral Tur\' an theorem on even cycles proved in \cite{cioabua2022evencycle, Nikiforovpaths, zhai2020spectralhexagon, ZW}. 

We end this section with some results that either follow from Theorems  \ref{theorem intersecting even cycles} and \ref{theorem intersecting 4-cycles}, or their proofs are almost identical to those given for Theorems \ref{theorem intersecting even cycles} and \ref{theorem intersecting 4-cycles}. 

\begin{remark}
\label{remark one on intersecting cycles and paths}
Observe that the intersecting even cycles $C_{2k_1, 2k_2, \ldots. 2k_t}$ also contains other bipartite graphs $H$ that are not contained in $S_{n, \kappa}^+$ or $F_{n,t}$ (depending on whether or not $H$ is a subgraph of an intersecting even cycle where at least one of the even cycles has more than four vertices).  This implies that for $n$ sufficiently large, $\rSPEX(n, C_{2k_1, 2k_2, \ldots. 2k_t}) = \rSPEX(n, H) \in \{\{S_{n, \kappa}^+\}, \{F_{n, t}\}\}$. Note that the size of the smallest color class of $H$ must also be $\kappa + 1$ (or $t+1$), for it to not be a subgraph of $S_{n, \kappa}^+$ (or $F_{n, t}$). 
\end{remark}

As an application of our observations in Remark \ref{remark one on intersecting cycles and paths}, we determine the spectral extremal graphs for the following families of graphs.
Let $\mathcal{CP}_{2k_1, \ldots, 2k_{t_1} ; 2p_1, \ldots, 2p_{t_2}}$ be the graph obtained by intersecting $t_1$ even cycles of lengths $2k_i \ge 4$, for $ i \le t_1$, and $t_2$ even paths on $2p_i$ vertices where $2p_i\ge 4$, for $i \le t_2$ at a unique vertex. Equivalently, we can obtain $\mathcal{CP}_{2k_1, \ldots, 2k_{t_1} ; 2p_1, \ldots, 2p_{t_2}}$ from $ C_{2k_1, \ldots, 2k_{t_1}, 2p_1, \ldots, 2p_{t_2}}$ by deleting one edge from each of the constituent $C_{2p_i}$, where the deleted edge was adjacent to the central vertex, for all $1 \le i \le t_2$.  
Clearly, $\mathcal{CP}_{2k_1, \ldots, 2k_{t_1} ; 2p_1, \ldots, 2p_{t_2}} \subset C_{2k_1, \ldots, 2k_{t_1}, 2p_1, \ldots, 2p_{t_2}}$. Then the following two results for intersecting even cycles and paths directly follow from Theorems \ref{theorem intersecting even cycles} and \ref{theorem intersecting 4-cycles}.
\begin{theorem}
\label{theorem intersecting even cycles and paths}
Let $\max\{k_1, k_2, \ldots, k_{t_1}, p_1, p_2, \ldots, p_{t_2}\} \ge 3$ where either $t_1 \ge 1$ or $t_2 \ge 2$, and \newline $\kappa: = \left(\sum_{i=1}^{t_1} k_i \right) + \left(\sum_{i=1}^{t_2} p_i\right) - (t_1 + t_2)$. If $G$ is a graph of sufficiently large order $n$ and
$\lambda(G) \ge \lambda(S_{n, \kappa}^{+})$ then $G$ contains $\mathcal{CP}_{2k_1, \ldots, 2k_{t_1} ; 2p_1, \ldots, 2p_{t_2}}$ unless $G = S_{n, \kappa}^{+}$.
\end{theorem}

\begin{theorem}
\label{theorem intersecting even cycles and paths 4 vert}
Let $k_i = p_j = 2$ for $1 \le i \le t_1$ and $t:= t_1 + t_2$. If $G$ is a graph of sufficiently large order $n$ and
$\lambda(G) \ge \lambda(F_{n, t})$ then $G$ contains $\mathcal{CP}_{2k_1, \ldots, 2k_{t_1} ; 2p_1, \ldots, 2p_{t_2}}$ unless $G = F_{n, t}$.
\end{theorem}

\begin{remark}
\label{remark two on intersecting cycles and paths}
The intersecting even cycles $C_{2k_1, \ldots, 2k_t}$ also contains several other connected bipartite subgraphs with a smallest color class consisting of exactly $\kappa + 1$ vertices, apart from those mentioned above. Some of these, however, are subgraphs of $S_{n, \kappa}^+$ (or $F_{n, t}$). For example, consider the graphs obtained by deleting at most one edge from every constituent cycle of $C_{2k_1, \ldots, 2k_t}$, where none of the deleted edges are adjacent to the central vertex. If at least one edge is deleted, such graphs are not contained in $S_{n, \kappa}$ but are contained in $S_{n, \kappa}^+$ (or $F_{n, t}$). 
\end{remark}

 \begin{remark}
 \label{remark three on intersecting cycles and paths}
 Let $H$ be a connected bipartite subgraph of $C_{4, 4, \ldots, 4}$, the intersecting even cycle consisting of $t$ cycles of length four that intersect at a unique vertex. If
 $H\subset F_{n, t}$, then $H \subset S_{n, t}^+$. To see this, first note that if $H \subset F_{n, t}$, it follows from Theorem \ref{theorem intersecting even cycles and paths 4 vert} that $H$ must be obtained from $C_{4, 4, \ldots, 4}$ by deleting at least one edge not adjacent to its center. Next, observe that any graph obtained from $C_{4, 4, \ldots, 4}$ by deleting an edge not adjacent to its center, must be contained in $S_{n, t}^+$ as well. 
    
 \end{remark}

It follows from Remarks \ref{remark two on intersecting cycles and paths} and \ref{remark three on intersecting cycles and paths} that if $H \subset C_{2k_1, 2k_2, \ldots, 2k_t}$, $H \subset S_{n, \kappa}^+$ (or $F_{n, t}$) and the size of the smallest color class of $H$ is $\kappa + 1$, then in fact, $H$ is not contained in $S_{n, \kappa}$. We are now ready to state our final theorem which has an almost identical proof to that for Theorems \ref{theorem intersecting even cycles} and \ref{theorem intersecting 4-cycles}. We therefore, do not include an independent proof for it, to avoid redundancy, but include additional remarks to describe slight modifications required for the same Lemmas to hold true. 

\begin{theorem}
\label{theorem small subgraph of intersecting even cycle}
   Let $k_1, k_2, \ldots, k_t \ge 2$ and $\kappa:= \left(\sum_{i=1}^t k_i\right) - t$. Let $H$ be a connected subgraph of $ C_{2k_1, 2k_2, \ldots, 2k_t}$ with a smallest color class consisting of $\kappa + 1$ vertices. If either $\max\{k_1, k_2, \ldots, k_t\} \ge 3$ and $H \subset S_{n, \kappa}^+$ or $k_1 = k_2 = \ldots = k_t$ and $H \subset F_{n, \kappa}$, then for sufficiently larger order $n$, $\rSPEX(n, H) = S_{n, \kappa}$. 
\end{theorem}

\section{Organization and Notation}

For an integer $k$, let $k':= k - 1$. In what follows, whenever we will be using $t$ positive integers $k_1, k_2, \ldots, k_t \ge 2$, to define an intersecting even cycle $C_{2k_1, 2k_2, \ldots, 2k_t}$, we will assume without loss of generality that $k_1 \le k_2 \le \ldots \le k_t$ and $k_1 \ge 2$. Also, let $\kappa  := \sum_{i=1}^t (k_i - 1) = \sum_{i=1}^t k_i' $. We wish to show that $\mathrm{SPEX}(n, C_{2k_1, 2k_2, \ldots 2k_t}) = \{S_{n, \kappa }^+$\} whenever we have $k_t \ge 3$ and $\mathrm{SPEX}(n, C_{2k_1, 2k_2, \ldots 2k_t}) = \{F_{n, t}\} = \{K_t \vee \left(\left\lfloor\frac{n-t}{2}\right\rfloor K_2 \cup \gamma K_1 \right)\}$, where $\gamma = (n-t) \pmod{2}$. 
 
For $n$ sufficiently large, let $G$ denote some graph in $\mathrm{SPEX}(n, C_{2k_1, 2k_2, \ldots, 2k_t})$. 

Let $\mathrm{x}$ be the scaled Perron vector whose maximum entry is equal to $1$ and let $z$ be a vertex having Perron entry $\mathrm{x}_z = 1$. We define positive constants $\eta, \epsilon$ and $\alpha$ satisfying 
\begin{align}\label{choice of constants}
\begin{split}
\eta & < \min\left\{\frac{1}{10 \kappa}, \frac{1}{(\kappa+t)} \cdot \left[\left( 1 - \frac{4}{5\kappa}\right)\left(\kappa - \frac{1}{16\kappa^2}\right) - (\kappa - 1)\right]\right\}\\
\epsilon &< \min\left\{\eta, \frac{\eta}{2}, \frac{1}{16\kappa
^3}, \frac{\eta}{32\kappa^3+2}\right\}\\
\alpha &< \min \left\{ \eta, \frac{\epsilon^2}{20\kappa}
\right\}.
\end{split}
\end{align}

We leave the redundant conditions in these inequalities so that when we reference this choice of constants it is clear what we are using. 

Let $L$ be the set of vertices in $V(G)$ having ``large" weights and $S = V \setminus L$ be the vertices of ``small" weights, \[L := \{v \in V(G) \mid \mathrm{x}_v \geq \alpha\}, S := \{v \in V(G) \mid \mathrm{x}_v < \alpha\}.\]
 Also, let $M \supset L$ be the set of vertices of ``not too small" a weight: \[M := \{v \in V(G) \mid \mathrm{x}_v\geq \alpha/4\}.\] Denote by $L' \subset L$ the following subset of vertices of ``very large'' weights: \[L' := \{v \in L \mid \mathrm{x}_v \geq \eta\}.\]

 For any set of vertices $A \subset V(G)$ and constant $\gamma$, let $A^\gamma = \{v \in A \mid \rx_v \ge \gamma\}$.

For any graph $H = (V, E)$, let $N_i(u)$ be the vertices at distance $i$ from some vertex $u \in V$. For any set of vertices $A \subset V$, let $A_i = A_i(u) = A \cap N_i(u)$. Thus, for the graph $G$, $L_i(u), S_i(u)$ and $M_i(u)$ are the sets of vertices in $L, S$ and $M$ respectively, at distance $i$ from $u$. If the vertex is unambiguous from context, we will use $L_i$, $S_i,$ and $M_i$ instead.  Let $d(u) := |N_1(u)|$ and $d_A(u) := |A_1(u)|$ denote the \textit{degree} of $u$ in $H$ and the number of neighbors of $u$ lying in $A$, respectively. In case the graph $H$ is not clear from context, we will use $d_H(u)$ instead of $d(u)$ to denote the degree of $u$ in $H$. Also, let $H[A]$ denote the induced subgraph of $H$ by the set of vertices $A$, and for two disjoint subsets $A, B \subset V$, let $H[A, B]$ denote the bipartite subgraph of $H$ induced by the sets $A$ and $B$ that has vertex set $A \cup B$ and edge set consisting of all edges with one vertex in $A$ and the other in $B$. 
Finally, let $E(A)$, $E(A,B)$ and $E(H)$ denote the set of edges of $H[A]$, $H[A,B]$ and $H$, respectively, and $e(A), e(A, B)$, and $e(H)$ denote the cardinalities $|E(A)|$, $|E(A,B)|$ and $|E(H)|$, respectively.   

In Section \ref{section preliminary lemmas} we give some background graph theory lemmas and prove preliminary lemmas on the spectral radius and Perron vector of spectral extremal graphs. We reiterate that every lemma applies to both the spectral extremal graphs when the forbidden graph is an intersecting cycle and when it is some $H \subset C_{2k_1, \ldots, 2k_t}$ as in the statement for Theorem \ref{theorem small subgraph of intersecting even cycle}. Observations specific to Theorem \ref{theorem small subgraph of intersecting even cycle} appear only in Remarks \ref{remark for intersectiong circuits in large complete bipartite graphs} and \ref{remark for G is connected}.
In Section \ref{section structural results for extremal graphs} we progressively refine the structure of our spectral extremal graphs. Using these results, we prove Theorems \ref{theorem intersecting even cycles} and \ref{theorem intersecting 4-cycles} in Section \ref{section main proof}.  We end with Section \ref{section minor free} giving some applications of Theorems \ref{theorem intersecting even cycles} and \ref{theorem intersecting 4-cycles} to the related minor-free spectral Tur\' an problem for intersecting even cycles.

\section{Lemmas from spectral and extremal graph theory}
\label{section preliminary lemmas}

In this section we record several lemmas that will be used subsequently. Some lemmas involve calculations that require $n$ to be sufficiently large, without stating this explicitly.
We begin this section by recording two known results.

\begin{lemma}\cite{cioabua2022evencycle}
\label{lower bound for spectral radius of nonnegative matrices}
For a non-negative symmetric matrix $B$, a non-negative non-zero vector $\mathbf{\mathrm{y}}$ and a positive constant $c$, if $B \mathbf{\mathrm{y}} \geq c \mathbf{\mathrm{y}}$ entrywise, then $\lambda(B) \geq c$. 
\end{lemma}

\begin{lemma}[Erd\H{o}s-Gallai \cite{gallai1959maximal}]
\label{extremal path}
Any graph on $n$ vertices with no subgraph isomorphic to a path on $\ell$ vertices has at most $\dfrac{(\ell-2)n}{2}$ edges.
\end{lemma}

It is easy to see that any bipartite graph is contained in a complete bipartite graph, where the color classes in either graphs have the same sizes. 
The following lemma is the version of this known result on intersecting even cycles, and we record it here to refer to subsequently.

\begin{lemma}
\label{intersectiong circuits in large complete bipartite graphs}
For $1 \le i \le t$, $k_i \ge 3$, and $\sum_{i = 1}^t k_i' = \kappa$, the bipartite graph $K_{\kappa + 1, \kappa + t}$ contains $C_{2k_1, \ldots, 2k_t}$. \end{lemma}
\begin{proof}
The intersecting cycle $C_{2k_1, \ldots, 2k_t}$ is a bipartite graph, whose smaller part (which contains the center of the intersecting cycle) has $\left(\sum_{i = 1}^t k_i'\right) + 1 = \kappa + 1$ vertices and larger part has $\kappa + t$ vertices.  
It is clear from this that $C_{2k_1, \ldots, 2k_t}  \subset K_{\kappa + 1,\kappa + t}$. 

\end{proof}

\begin{remark}
\label{remark for intersectiong circuits in large complete bipartite graphs}
    Note that, if $H$ is a subgraph of $C_{2k_1, \ldots, 2k_t}$, with $\kappa$ as defined above, then $H$ is also contained in $K_{\kappa + 1, \kappa + t}$. 
\end{remark}

The following lemma will be used to show that any  $C_{2k_1, \ldots, 2k_t}$-free graph $G$ cannot contain a large path $P$ such that $E(P) \subseteq E(G[N_1(v)] \cup G[N_1(v), N_2(v)])$, for any vertex $v$ of $G$.   

\begin{lemma}
\label{long paths in bipartite subgraph}
Let $1 \le i \le t$, $k_i \ge 2$, with $\sum_{i = 1}^t k_i' = \kappa$ and $H = (V, E)$ be a graph whose vertices are partitioned into two sets $U$ and $W$. If $E = E(H[U]) \cup E(H[U, W])$, and $H$ has a path on $4\kappa  + t$ vertices, then $H$ has $t$ disjoint paths of order $2k_i' + 1$, $1 \le i \le t$, and all the endpoints of the paths are in $U$.
\end{lemma}
\begin{proof}
We will prove the lemma by first proving the following claim.
\begin{claim}
If $H$ contains a path with $4k_i' + 1$ vertices then $H$ must contain a path of order $2k_i' + 1$ with both end points in $U$.
\end{claim}
\begin{proof}
If $H$ has a path on $4k_i' + 1$ vertices, then there must be a subpath $v_1 v_2 \ldots v_{4k_i'}$ of length $4k_i'$ where we may assume without loss of generality that $v_1 \in U$. 
We proceed to show that $v_1 v_2 \ldots v_{4k_i'}$ contains a path $v_{i_0} v_{i_0 + 1} \ldots v_{i_0 + 2k_i'}$ of order $2k_i' + 1$ with both endpoints in $U$ via contradiction. To do this assume to the contrary that there is no such subpath of order $2k_i' + 1$ with both endpoints in $U$.

Since $v_1 \in U$, we have that $v_{2k_i'+1} \in W$. Then both $v_{2k_i'}, v_{2k_i' + 2} \in U$. Now, if $v_2 \in U$ then we have a contradiction that $v_2 \ldots v_{2k_i'+2}$ is a path of order $2k_i' + 1$ with both endpoints in $U$. So assume $v_2 \in W$ and therefore $v_3 \in U$. Thus $v_{2k_i'+3} \in W$ and $v_{2k_i'+4} \in U$. 
Proceeding similarly we get \[v_4 \in W, v_5 \in U, v_{2k_i' + 5} \in W, \ldots, v_{2k_i' - 2} \in W, v_{2k_i' - 1} \in U, v_{4k_i' - 1} \in W, v_{4k_i'} \in U, \text{ and finally } v_{2k_i'} \in W,\] a contradiction. 
Thus there must exist a path $v_{i_0} v_{i_0 + 1} \ldots v_{i_0 + 2k_i'}$ in $v_1 v_2 \ldots v_{4k_i'}$ with both endpoints in $U$. 
\end{proof}
 
Next assume that $H$ contains a path $v_1 v_2 \ldots v_{4\kappa + t}$ of order $4\kappa  + t$. Then using our previous claim we have that the path $v_1 v_2 \ldots v_{4k_1' + 1}$ must have a subpath with $2k_1' + 1$ vertices, both of whose endpoints are in $U$. Similarly, the path $v_{4k_1' + 2} v_{4k_1' + 3} \ldots v_{4k_1' + 4k_2' + 2}$, must contain a subpath with $2k_2' + 1$ vertices, both of whose endpoints are in $U$. 
In general, we have that the path $v_{4(\sum_{j=1}^{i-1} k_j') + i} v_{4(\sum_{j=1}^{i-1} k_j') + i + 1} \ldots v_{4(\sum_{j=1}^{i} k_j') + i}$  must contain a subpath with $2k_i' + 1$ vertices both of whose endpoints are in $U$, for all $1 \le i \le t$.

Thus, the statement of the lemma follows. 
\end{proof}

We can combine Lemmas \ref{extremal path} and \ref{long paths in bipartite subgraph} to get the following Lemma for disjoint paths in a subgraph.
\begin{lemma}
\label{Many edges in a bipartition II}
Suppose that $t \geq 1$, $k_i' \ge 1$ for all $1 \le i \le t$, and $\kappa = \sum_{i=1}^t k_i'$. For a graph $\hat{H}$ on $n$ vertices, let $U \sqcup W$ be a partition of its vertices.
If
    \begin{equation}
    \label{eqn Many edges in a bipartition A II}
        e(U) + e(U, W) > \frac{\left(4\kappa +t -2\right)}{2}n
        \end{equation}
then there exist $t$ disjoint paths of orders $2k_i' + 1$, for $1 \le i \le t$, with both ends in $U$.    

Further, if 
\begin{equation}
\label{eqn Many edges in a bipartition B}
        2e(U) + e(U, W) > \frac{\left(4\kappa +t -2\right)}{2}(|U| + n)
         \end{equation}
then there exist $t$ disjoint paths of orders $2k_i' + 1$, for $1 \le i \le t$, with both ends in $U$.
\end{lemma}

\begin{proof}
Let $H$ be the subgraph of $\hat{H}$ on the same vertex set $U \sqcup W$ with edge set $E(U) \sqcup E(U, W)$. Then by applying Lemma \ref{extremal path} on $H$ and the assumption $e(U) + e(U, W) > \frac{\left(4\kappa +t -2\right)}{2}n$, we have that $H$ must contain a path of order $4\kappa + t$. Applying Lemma \ref{long paths in bipartite subgraph} on the subgraph $H$ of $\hat{H}$ which is partitioned into two sets $U$ and $W$, we have that $H$ and therefore $\hat{H}$ contains $t$ disjoint paths of orders $2k_i' + 1$ for all $1 \le i \le t$.

Similarly, if $e(U) > \frac{\left(4\kappa +t -2\right)}{2}|U|$, we have that $\hat{H}[U]$ and therefore $\hat{H}$ contains $t$ disjoint paths of orders $2k_i' + 1$ for all $1 \le i \le t$. Thus, if $\hat{H}$ does not contain $t$ disjoint paths of orders $2k_i' + 1$ for all $1 \le i \le t$, we must have that $2e(U) + e(U, W) = e(U) + e(U) + e(U, W) \le \frac{\left(4\kappa +t -2\right)}{2}(|U| + n)$.
\end{proof}

The following result on the sum of degree squares of a graph generalizes Theorem 2 of \cite{nikiforov2009degree} to the case where the forbidden graphs are intersecting even cycles.
\begin{theorem}
\label{theorem degree squares}
Let $G$ be a graph with $n$ vertices and $m$ edges. If $G$ does not contain a $C_{2k_1, 2k_2, \ldots, 2k_t}$ where $k_1 \ge 2$, then
\begin{equation}
    \sum_{u \in V(G)} d_G^2(u) < (4\kappa + t)(n-1)n.
\end{equation}
\end{theorem}

\begin{proof}
Let $u$ be any vertex of $G$. Now 

    \begin{equation}
    \begin{split}
    \sum_{v \in N_1(u)} \left(d(v) - 1\right) = 2e(G[N_1(u)]) + e(G[N_1(u), N_2(u)])  
    &\le \frac{(4\kappa +t -2)}{2}\left(|N_1(u)| + (|N_1(u)| + |N_2(u)|)\right)\\
    & \le (2\kappa + t/2 - 1)(d(u) + n - 1).
\end{split}
\end{equation}
Thus,

\begin{equation}
    \begin{split}
        \sum_{u \in V(G)}d_G^2(u) = \sum_{u \in V(G)}\sum_{v \in N_1(u)} d(v) &\le \sum_{u \in V(G)}((2\kappa + t/2 - 1)(d(u) + n-1) + d(u))\\ 
        & \le (4\kappa + t )e(G) + (2\kappa + t/2 - 1)(n-1)n\\
        & < (4\kappa + t)(n-1)n,
    \end{split}
\end{equation}
where the last inequality uses $e(G) \le \binom{n}{2}$.
\end{proof}

The following result follows from the works of F\"uredi \cite{furedi1991turan}, and Alon, Krivelevich and Sudakov \cite{alon2003turan}. The latter is recorded as Theorem \ref{alonturan} above.
\begin{lemma}
\cite{alon2003turan, furedi1991turan}
\label{extremal containing C4}
Let $H$ be a bipartite graph with maximum degree at most two on one side, then there exists a positive constant $C$ such that 
\[\mathrm{ex}(n, H) \le Cn^{3/2}.\]
\end{lemma}
The following is a result of Conlon and Lee which implies that equality occurs in Lemma \ref{extremal containing C4} if and only if $H$ contains a $C_4$. We record this as it appears in Theorem 1.3 of \cite{conlon2021extremal}.
\begin{lemma}
\cite{conlon2021extremal}
\label{bounds on turan number for 1-subdivisions}
For any bipartite graph $H$ with maximum degree two on one side containing no $C_4$, there exist positive constants $C$ and $\delta$ such that 
\[\mathrm{ex}(n, H) = Cn^{3/2 - \delta}.\]
\end{lemma}

In particular, if we apply Lemma \ref{bounds on turan number for 1-subdivisions} to intersecting even cycles without any $C_4$, we get the following lemma.
\begin{lemma}
\label{extremal not containing C4}
For any intersecting cycle $C_{2k_1, 2k_2, \ldots, 2k_t}$ containing no $C_4$, there exist positive constants $C$ and $\delta$ such that
\begin{equation}
    \mathrm{ex}(n, C_{2k_1, 2k_2, \ldots, 2k_t}) = Cn^{3/2 - \delta}.
\end{equation}
\end{lemma}

Next we determine bounds for the spectral radius of a spectral extremal graph $G$.

\begin{lemma}\label{bounds on lambda}
$\sqrt{\kappa  n} \leq \dfrac{\kappa -1 + \sqrt{(\kappa -1)^2 + 4\kappa (n-\kappa )}}{2} \leq \lambda(G) \leq \sqrt{(4\kappa + t)(n-1)} < \sqrt{5\kappa n}$.
\end{lemma}
\begin{proof}
The lower bound is obtained from the spectral radius of $S_{n,\kappa}$. The upper bound is obtained by using the second-degree eigenvalue-eigenvector equation along with  Lemma \ref{bounds on turan number for 1-subdivisions}:

Let $u \in V(G)$. We use Lemma \ref{Many edges in a bipartition II} over the graph $G[N_1(u) \cup N_2(u)]$ with $U = N_1(u)$ and $W = N_2(u)$. We know that $|N_1(u)| = d(u)$. Also, there cannot be $t$ disjoint paths on $2k_i'+1$ vertices, for $1 \le i \le t$ in $G[N_1(u) \cup N_2(u)]$ with both end points in $N_1(u)$, else $G[\{u\} \cup N_1(u) \cup N_2(u)]$ contains a $C_{2k_1, 2k_2, \ldots, 2k_t}$ with center $u$. 
So \eqref{eqn Many edges in a bipartition B} implies that 
\begin{equation}
\label{eqn precursor to bounds on spectral radius using bipartition edges}
\begin{split}
2e(N_1(u)) + e(N_1(u), N_2(u)) &\le \frac{\left(4\kappa +t -2\right)}{2}(2 |N_1(u)| + |N_2(u)|)\\ 
&\le (2\kappa + t/2 -1)(d(u) + n - 1).    
\end{split}    
\end{equation}
The spectral radius of a non-negative matrix is at most the maximum of the row-sums of the matrix. Applying this result for $A^2(G)$ and its spectral radius $\lambda^2$ and using \eqref{eqn precursor to bounds on spectral radius using bipartition edges}, we obtain that
\begin{equation}
\begin{split}
\lambda^2 &\leq \max_{u \in V(G)} \big\{\sum_{w\in V(G)}\ A^2_{u,w} \big\} = \max_{u \in V(G)}\big\{\sum_{v \in N(u)}d(v)\big\}\\ &= \max_{u \in V(G)}\big\{d(u) + 2e(N_1(u)) + e(N_1(u), N_2(u))\}\\
&\leq (2\kappa + t/2)\left(\max_{u \in V(G)}\big\{d(u)\big\}\right) + (2\kappa + t/2 -1)(n - 1)\\
&< (4\kappa + t)(n-1).
\end{split}
\end{equation}
Thus, $\lambda < \sqrt{(4\kappa + t)(n-1)} < \sqrt{5\kappa n}$, since $t \le \sum_{i=1}^t k_i' = \kappa$.
\end{proof}

We will now determine upper bounds for the number of vertices in the sets $L$ and $M$ that consist of vertices with relatively large weights.  
\begin{lemma}\label{lemma:LM}
There exists a positive constant $\delta$ such that $|L| \le \dfrac{n^{1 - \delta}}{\alpha}$ and $|M| \le \dfrac{4n^{1 - \delta}}{\alpha}$.
\end{lemma}
\begin{proof}
We will break the proof into two cases depending on whether or not there exists a positive constant $\delta$ such that $\mathrm{ex}(n, C_{2k_1, \ldots 2k_t}) \le \dfrac{\sqrt{\kappa}}{2}n^{3/2 - \delta}$ (that is, depending on whether or not $C_{2k_1, \ldots, 2k_t}$ contains a $C_4$). The proof for the second case also proves the first case without any $C_4$. However we include both proofs to exhibit the need for more delicate techniques when $\mathrm{ex}(n,F) = \Theta(n^{3/2})$. 

\begin{case}
$\mathrm{ex}(n, C_{2k_1, \ldots, 2k_t}) \le \dfrac{\sqrt{\kappa}}{2}n^{3/2 - \delta}$. 
\end{case}
For any $v \in V$, we have that 
\begin{equation}
\label{eigenvalue-eigenvector equation}
    \lambda \mathrm{x}_v = \sum_{u \sim v} \mathrm{x}_u \leq d(v).
\end{equation}
 Combining this with the lower bound in Lemma \ref{bounds on lambda} gives $\sqrt{\kappa n} \mathrm{x}_v \leq \lambda \mathrm{x}_v \leq d(v)$. Summing over all vertices $v \in L$ gives
\begin{equation}
 |L|\sqrt{\kappa n}\alpha \leq \sum_{v\in L} d(v) \leq \sum_{v \in V} d(v) \leq 2 e(G) \leq 2 \mathrm{ex}(n, C_{2k_1, 2k_2, \ldots, 2k_t}) \leq \sqrt{\kappa}n^{3/2 - \delta}.
\end{equation}
Thus, $|L| \leq \dfrac{n^{1 - \delta}}{\alpha}$. A similar argument shows that $|M| \leq \dfrac{4n^{1 - \delta}}{\alpha}$.

\begin{case}
$\mathrm{ex}(n, C_{2k_1, \ldots, 2k_t}) = \Theta(n^{3/2})$.    
\end{case}

Let $G \in \mathrm{SPEX}(n, C_{2k_1, 2k_2, \ldots, 2k_t})$ where $2 = k_1 \le k_2 \le \ldots \le k_t$. If $t = 1$, then $C_{2k_1} = C_4$ since we have assumed that $\rex(n, C_{2k_1}) = \Theta(n^{3/2})$. Also, It is known from \cite{Nikiforovpaths, ZW} that $G \cong F_{n, 1}$. Hence, $|L| = 1$ and we are done. In what follows we will assume that $t \ge 2$.
We will show that $|L| \le n^{3/4} < \dfrac{n^{1-\delta}}{\alpha}$.

Assume to the contrary that $|L| > n^{3/4}$. Every vertex in $L$ has degree at least $\sqrt{\kappa n} \alpha$. 
We know by our observations in \eqref{AKS bounds} and Lemma \ref{extremal containing C4} that there is some positive constant $C$ such that $e(G) \le \mathrm{ex}(n, C_{2k_1, 2k_2, \ldots, 2k_t}) \le Cn^{3/2}$. Thus, at most $n^{3/4}$ vertices in $L$ have degree at least $2Cn^{3/4}$ and consequently, at least $|L|- n^{3/4}$ vertices of $L$ have degrees less than $2Cn^{3/4}$. Let $\mathbb{L}$ denote this subset of vertices in $L$ with degrees less than $2Cn^{3/4}$ and $\mathcal{L}$ the set of remaining vertices in $L$ with degrees at least $2Cn^{3/4}$. If $|L| > n^{3/4}$, we will show a contradiction by proving that $\mathbb{L} = \emptyset$. To this end, let us assume there is a vertex $v \in \mathbb{L}$. Then $d(v)  < 2Cn^{3/4}$. Now let $H$ denote the graph with vertex set $N_1(v) \cup N_2(v)$ and  
$E(H) = E(N_1(v), N_2(v))$. To make it easier to read, we will use $N_i, S_i, L_i, \mathbb{L}_i$ and $\mathcal{L}_i$ to denote the set of vertices in $V(G), S, L, \mathbb{L}$ and $\mathcal{L}$, respectively, at distance $i$ from $v$.
Then \begin{equation}
    \begin{split}
        \kappa n \alpha \le \kappa n \rx_v &\le \lambda^2 \rx_v = \sum_{u \sim v}\sum_{w \sim u} \mathrm{\bx}_w \le d(v)\rx_v + 2e(N_1) + \sum_{u \sim v}\sum_{\substack{w \sim u\\ w \in N_2}} \rx_w\\
        &\le d(v) + 2e(N_1) + e(N_1, \mathcal{L}_2) + \sum_{u \sim v}\sum_{\substack{w \sim u\\ w \in N_2 \setminus \mathcal{L}_2}} \rx_w\\
        &\le |N_1| + \frac{(4\kappa + t - 2)}{2}(|N_1| + (|N_1| + |\mathcal{L}_2|)) + \sum_{u \sim v}\sum_{\substack{w \sim u\\ w \in N_2 \setminus \mathcal{L}_2}} \rx_w\\
        &\le C_1 n^{3/4} + \sum_{u \sim v}\sum_{\substack{w \sim u\\ w \in N_2 \setminus \mathcal{L}_2}} \rx_w,
    \end{split}
\end{equation}
for some positive constant $C_1 \ge \frac{(4\kappa + t)}{2}(4C+1)$.

Thus, $e(N_1, N_2 \setminus \mathcal{L}_2) \ge  \kappa n \rx_v -  C_1 n^{3/4} \ge (\kappa-\epsilon_1)n \rx_v$, where we may take $\epsilon_1 > 0$ as small as required by taking $n$ sufficiently large.

For any constant $\gamma$, let $\bbL^\gamma$ denote the set of vertices $\{v \in \bbL \mid \mathrm{\bx}_v \ge \gamma\}$.  Further, we use $\bbL^\gamma_i$ to denote the set of vertices in $\bbL^{\gamma}$ at distance $i$ from $v$. Finally, set $\sigma : = \dfrac{\alpha}{4\kappa}$. 

In the following, we will show that $|\bbL^{\gamma}| = 0$ for all $\gamma \ge \alpha.$  First we will show that if $|\bbL^{\gamma}| = 0$ for some $\gamma$, then $|\bbL^{\gamma - \sigma}| = 0$. Clearly, since $|\bbL^{\gamma}| = 0$ for all $\gamma > 1$, this will inductively give us that $\bbL^{\alpha} = \bbL = \emptyset$.
To do this we will show the existence of disjoint paths $P_{2k_i'+1}$ for $1\le i \le t$ in $H$, with both end points in $N_1$.

Given $\bbL^{\gamma} = \emptyset$, recall that for $v \in \bbL^{\gamma - \sigma}$, 
\begin{equation}
    \begin{split}
        \kappa n (\gamma - \sigma) &\le \kappa n \rx_v \le \lambda^2 \rx_v = \sum_{u \sim v}\sum_{w \sim u} \rx_w \le C_1 n^{3/4} + \sum_{u \sim v}\sum_{\substack{w \sim u\\ w \in \bbL_2^{\gamma}}} \rx_w + \sum_{u \sim v}\sum_{\substack{w \sim u\\ w \in N_2 \setminus (\bbL_2^{\gamma} \cup \mathcal{L}_2})} \rx_w\\
        &\le C_1 n^{3/4} + \sum_{u \sim v}\sum_{\substack{w \sim u\\ w \in N_2 \setminus (\bbL_2^{\gamma} \cup \mathcal{L}_2)}} \gamma =  C_1 n^{3/4} + e(N_1, N_2 \setminus (\bbL_2^{\gamma} \cup \mathcal{L}_2))\gamma.
    \end{split}
\end{equation}

Since $\sigma : = \dfrac{\alpha}{4\kappa}$,
\begin{equation}
    \begin{split}
      \left( \dfrac{2\kappa - 0.5}{2}\right)n = \kappa n \left(1 - \dfrac{1}{4\kappa}\right) \le
       \kappa n \left(1 - \dfrac{\sigma}{\gamma}\right) = \kappa n \left(\dfrac{\gamma - \sigma}{\gamma}\right) \le C_2 n^{3/4} + e(N_1, N_2 \setminus (\bbL_2^{\gamma} \cup \mathcal{L}_2))
    \end{split}
\end{equation}

for some positive constant $C_2 = C_1/\gamma$.
For $n$ sufficiently large, this gives 
\begin{equation}
    \begin{split}
      \dfrac{(2\kappa - 0.6)}{2}n \le e(N_1, N_2 \setminus (\bbL_2^{\gamma} \cup \mathcal{L}_2)).
    \end{split}
\end{equation}
Thus, by Lemma \ref{extremal path} there is a $P_{2\kappa + 1}$ in $H$, and therefore a $P_{2\kappa - 1}$ in $H$ with both end points in $N_1$. 

Since $\kappa = \sum_{i=1}^t k_i'$, $k_1' = 1$, and $t \ge 2$, it implies that $k_t' + 1 \le \kappa$. Thus, $2k_t' + 1 \le 2\kappa - 1$, and there must exist $P_{2k_i' + 1}$ in $H$ with both end points in $N_1$ for all $1 \le i \le t$. It remains to show that we can find disjoint paths $P_{2k_i' + 1}$ in $H$ with both end points in $N_1$ for all $1 \le i \le t$.

To do this we proceed inductively. Say we are able to show the existence of the first $j$ disjoint paths,  $P_{2k_i' + 1}$ where $1 \le i \le j$ for some $j$ such that $1 \le j < t$, in $H$ with both end points in $N_1$. We will next show by induction that there is a disjoint $P_{2k_{j+1}' + 1}$ as well, oriented such that both its end points lie in $N_1$. Let $H_{j+1}$ be the induced subgraph of $H$ with vertex set $V(H) \setminus (\mathcal{L}_2 \cup_{i=1}^j V(P_{2k'_i+1}))$. Now, for some positive constant $c$, if $W$ is any subset of $V(H)$ satisfying $|W| \le c$, then $\sum_{u \in W}\sum_{w \sim u} \rx_w = \sum_{u \in W} \lambda \rx_u \le c\lambda$. Consequently,
\begin{equation}
\begin{split}
        \kappa n (\gamma - \sigma) & \le \lambda^2 \rx_v = \sum_{u \sim v}\sum_{w \sim u} \mathrm{\bx}_w \le  C_1 n^{3/4} + 
         \sum_{u \sim v}\sum_{\substack{w \sim u\\ w \in N_2 \setminus (\bbL_2^{\gamma} \cup \mathcal{L}_2})} \rx_w\\
& \le C_1 n^{3/4} + \sum_{\substack{u \sim v \\ u \in W}}\sum_{\substack{w \sim u\\ w \in N_2 \setminus (\bbL_2^{\gamma} \cup \mathcal{L}_2})} \rx_w + \sum_{\substack{u \sim v\\ u \not\in W}}\sum_{\substack{w \sim u, w \in W\\ w \in N_2 \setminus (\bbL_2^{\gamma} \cup \mathcal{L}_2)}} \rx_w + \sum_{\substack{u \sim v \\ u \not\in W}}\sum_{\substack{w \sim u, w \not\in W\\ w \in N_2 \setminus (\bbL_2^{\gamma} \cup \mathcal{L}_2)}} \rx_w \\
& \le C_1 n^{3/4} + c \lambda + c(2Cn^{3/4}) \gamma + \sum_{\substack{u \sim v \\ u \not\in W}}\sum_{\substack{w \sim u, w \not\in W\\ w \in N_2 \setminus (\bbL_2^{\gamma} \cup \mathcal{L}_2)}} \gamma\\
& \le C_3 n^{3/4} + \sum_{\substack{u \sim v \\ u \not\in W}}\sum_{\substack{w \sim u, w \not\in W\\ w \in N_2 \setminus (\bbL_2^{\gamma} \cup \mathcal{L}_2)}} \gamma,
\end{split}
\end{equation}  
for some positive constant $C_3$ and $n$ taken to be sufficiently large. In particular, if $W$ is the set of vertices of $V(H)$, $W = \cup_{i=1}^j V(P_{2k'_j + 1})$ and $c = \sum_{i=1}^j (2k_i' + 1)$, we have 
\[\kappa n (\gamma - \sigma) \le  C_3 n^{3/4} + \sum_{\substack{u \sim v \\ v \not\in W}}\sum_{\substack{w \sim v, w \not\in W\\ w \in N_2 \setminus (\bbL_2^{\gamma} \cup \mathcal{L}_2)}} \gamma = C_3 n^{3/4} + e(H_{j+1})\gamma.\]

Thus, $e(H_{j+1}) > \dfrac{2\kappa -1}{2}n$ and there must exist a path $P_{2k_{j+1}' + 1}$ in $H_{j+1}$ with both end points in $N_1$. By induction, there exist all $t$ disjoint paths of lengths $2k_j'$ in $H$ with both end points in $N_1$, for all $1 \le j \le t$. These disjoint paths create the intersecting even cycle $C_{2k_1, 2k_2, \ldots, 2k_t}$ along with $v$, the central vertex. This is a contradiction, hence $\bbL^{\gamma -\delta} = \bbL^{\gamma} = \emptyset$. Proceeding similarly, we can show that $\bbL^{\gamma -\delta} = \bbL^{\gamma -2\delta} = \ldots = \bbL^{\alpha} (= \bbL)$, thus $|\bbL| = 0$, a contradiction. Therefore, we must have that $|L| \le n^{3/4} < \dfrac{n^{1-\delta}}{\alpha}$. Similarly, we can prove that $|M| \le \dfrac{4n^{1-\delta}}{\alpha}$.

\end{proof}

In our final lemma for this section, we determine lower bounds for the Perron entries in terms of the spectral radius of $G$.

\begin{lemma}
\label{G is connected}
The graph $G$ is connected and for any vertex $v \in V(G)$, we have that 
\begin{equation*}
\mathrm{x}_v \geq \dfrac{1}{\lambda(G)} > \dfrac{1}{\sqrt{5\kappa n}}.
\end{equation*}
\end{lemma}
\begin{proof}

Say $z$ is a vertex with maximum eigenweight $\mathrm{x}_z = 1$. Then $\sqrt{\kappa n} \le \lambda \le d(z)$. Now suppose to the contrary that $G$ is not connected. Say $V_1 \ni z$ and $V_2$ are two connected components of $G$. Then $\lambda(G[V_1]) = \lambda$. Next, for any $u \in V_2$, let $\hat{G}$ be the graph obtained from $G$, having an identical vertex set and $E(\hat{G}) = E(G[V_1]) \cup \{uz\}$. Then, $G[V_1]$ is a proper subgraph of $\hat{G}$. Additionally, $\hat{G}[V_1 \cup \{u\}]$ is connected and contains $G[V_1]$ as a proper subgraph. 
Therefore $\lambda < \lambda(\hat{G}[V_1 \cup \{u\}]) = \lambda(\hat{G})$.
Since $G \in \rSPEX(n, C_{2k_1, \ldots, 2k_t})$, we must have that $C_{2k_1, \ldots, 2k_t} \subset \hat{G}$. 

*Now, $u$ must be a part of the $C_{2k_1, \ldots, 2k_t}$ created in $\hat{G}$. 
However, $d_{\hat{G}}(u) = 1$, so $u$ cannot be in any $C_{{2k_1, \ldots, 2k_t}}$, and $\hat{G}$ is $C_{{2k_1, \ldots, 2k_t}}$-free.
But this contradicts the fact that $G \in \mathrm{SPEX}(n, C_{2k_1, \ldots, 2k_t})$. Thus, $G$ must be a connected graph. [Note the remark given after the end of the proof, to see a slight modification to the arguments in this paragraph* for purposes of Theorem \ref{theorem small subgraph of intersecting even cycle}].

Now let $v$ be an arbitrary vertex of $G$. If $v$ is adjacent to any vertex with eigenweight $1$, then clearly $\lambda \mathrm{x}_v \ge 1$ and therefore $\mathrm{x}_v \ge \frac{1}{\lambda}$ and we are done. 
However, if $v$ is not adjacent to $z$ and assume to the contrary that $\mathrm{x}_v < \frac{1}{\lambda(G)}$, then we modify $G$ to obtain the graph $G'$ as follows. Let $V(G') = V(G)$, and $E(G') = (E(G) \setminus \{uv: uv \in E(G)\}) \cup \{zv\}$. Then by the same arguments as above for $\hat{G}$ we can show that $G'$ is $C_{2k_1, \ldots, 2k_t}$-free. Moreover, $\mathrm{x}^TA(G')\mathrm{x} - \mathrm{x}^TA(G)\mathrm{x} = 2(1 - \lambda \mathrm{x}_v)\rx_v > 0$, so $\lambda(G') > \lambda(G)$, a contradiction. Hence, $\mathrm{x}_v \ge \frac{1}{\lambda}$.
\end{proof}

\begin{remark}
\label{remark for G is connected}
 Note the following version of the arguments in the proof of Lemma \ref{G is connected} that is relevant for the purposes of Theorem \ref{theorem small subgraph of intersecting even cycle}:

*Now if $\hat{G}$ contains a copy of $H\subset C_{2k_1, \ldots, 2k_t}$, then $u$ must be a part of the $H$ created in $\hat{G}$. 
However, $d_{\hat{G}}(u) = 1$ and $d_{G}(z) \ge \lambda > 2\kappa + t + 1 = |V(C_{2k_1, \ldots, 2k_t})|$ for $n$ large enough. So, if $\hat{G}$ contains a copy of $H$, then $G$ must also have contained a copy of $H$, a contradiction.   
Thus, $G$ must be a connected graph. 
 
\end{remark}

\section{Structural results for extremal graphs}
\label{section structural results for extremal graphs}

In this section, we continue to revise our estimations for the structure of $G$. We continue our use of auxiliary constants $\alpha, \epsilon,$ and $ \eta$. We assume $n$ to be large enough for all lemmas in this section. 

We begin by showing that the degree of any vertex in $L$ is linearly growing with respect to $n$ and therefore using Theorem \ref{theorem degree squares} it follows that $L$ has constant size.

\begin{lemma}
\label{degrees of vertices in L}
Any vertex $v \in L$ has degree $d(v) \ge \dfrac{\alpha}{10(4\kappa + t - 1)}n$. Moreover, $|L| < \dfrac{100(4\kappa + t)^3}{\alpha^2} < \dfrac{12500\kappa^3}{\alpha^2}$.
\end{lemma}

\begin{proof}
Assume to the contrary that there exists a vertex $v \in L$ of degree $d(v)<\dfrac{\alpha}{10(4\kappa + t - 1)}n$. Let $\mathrm{x}_v = c \ge \alpha$. The second degree eigenvector-eigenvalue equation with respect to the vertex $v$ gives that
\begin{equation}
\label{lower bound on degrees in L}
    \begin{split}
        \kappa nc \leq \lambda^2 c = \lambda^2 \mathrm{x}_v = \sum_{\substack{u \sim v \\ w \sim u}} \mathrm{x}_w 
        & \leq d(v)c + 2e(N_1(v)) + \sum_{u \sim v}\sum_{\substack{w \sim u, \\ w \in N_2(v)}}\mathrm{x}_w\\
        & \leq (4\kappa + t - 1)d(v) + \sum_{u \sim v}\sum_{\substack{w \sim u, \\ w \in M _2(v)}}\mathrm{x}_w + \sum_{u \sim v}\sum_{\substack{w \sim u, \\ w \in N_2 \setminus M_2(v)}}\mathrm{x}_w,
    \end{split}
\end{equation}
where the last inequality follows from $e(N_1(v)) \le (2\kappa + t/2 - 1) d(v)$.
Since $d(v) < \dfrac{\alpha}{10(4\kappa + t - 1)}n$ by our assumption and $c \ge \alpha$, we have
\begin{equation}
    (\kappa - 0.1)nc <  \sum_{u \sim v}\sum_{\substack{w \sim u, \\ w \in M _2(v)}}\mathrm{x}_w + \sum_{u \sim v}\sum_{\substack{w \sim u, \\ w \in N_2 \setminus M_2(v)}}\mathrm{x}_w.
\end{equation}
From $d(v)<\frac{\alpha}{10(4\kappa + t - 1)}$ and Lemma \ref{lemma:LM}, we get that
 \[\sum_{u \sim v}\sum_{\substack{w \sim u, \\ w \in M_2(v)}}\mathrm{x}_w \leq e(N_1(v), M_2(v)) \leq (2\kappa + t/2 - 1)(d(v) + |M|) < (2\kappa + t/2 - 1)\left(\dfrac{\alpha}{10(4\kappa + t - 1)}n + \dfrac{4n^{1 - \delta}}{\alpha}\right).\]
For $n$ sufficiently large, we have that 
\[
(2\kappa + t/2 - 1)\left(\dfrac{\alpha}{10(4\kappa + t - 1)}n + \dfrac{4n^{1 - \delta}}{\alpha}\right) \leq .2n\alpha \leq .2nc,
\]
and consequently using Lemma \ref{bounds on turan number for 1-subdivisions},
\[(\kappa - 0.3)nc < \sum_{u \sim v}\sum_{\substack{w \sim u, \\ w \in N_2 \setminus M_2(v)}}\mathrm{x}_w \leq e(N_1(v), N_2 \setminus M_2(v))\dfrac{\alpha}{4} \leq (2\kappa + t/2 -1)n \dfrac{\alpha}{4}.\]
This is a contradiction because $\kappa \geq t \ge 1$ and $c \ge \alpha$. Hence, $d(v) \geq \dfrac{\alpha}{10(4\kappa + t - 1)}n$ for all $v \in L$. 
Thus, for $n$ sufficiently large we have that $d^2(v) \ge \left(\dfrac{\alpha}{10(4\kappa + t - 1)}n\right)^2$ for all $v \in L$.

Combined with Theorem \ref{theorem degree squares} this gives us that
\[|L|\left(\dfrac{\alpha}{10(4\kappa + t - 1)}n\right)^2 \le\sum_{v \in L}d^2(v) \le \sum_{v \in V(G)}d^2(v) < (4\kappa + t)(n-1)n.\]

Therefore,
\[|L| \leq \dfrac{(4\kappa + t)(n-1)n(10(4\kappa + t - 1))^2}{\alpha^2n^2} < \dfrac{100(4\kappa + t)^3}{\alpha^2} \le \dfrac{12500 \kappa^3}{\alpha^2},\]
where we have used $t \le \kappa$ in the last inequality.
\end{proof}

We will now refine our arguments in the proof above to improve degree estimates for the vertices in $L'$.

\begin{lemma}
\label{mindegrees for vertices in L'}
If $v$ is a vertex in $L'$ with $\mathrm{x}_v = c$, then $d(v) \ge cn -\epsilon n$. 
\end{lemma}
\begin{proof}
Our proof is again by contradiction. The second degree eigenvalue-eigenvector equation with respect to the vertex $v$ gives
\begin{equation}
\begin{split}
    knc 
    & \le \lambda^2 c = \sum_{u \sim v} \sum_{w \sim u} \mathrm{x}_w = d(v) c + \sum_{u \sim v} \sum_{\substack{w \sim u,\\ w \neq v}} \mathrm{x}_w\\
    & \le d(v) c + \sum_{u \in S_1} \sum_{\substack{w \sim u \\ w \in L_1 \cup L_2}} \mathrm{x}_w + 2e(S_1)\alpha + 2e(L) + e(L_1, S_1)\alpha + e(N_1, S_2)\alpha.
\end{split}
\end{equation}
The observation that the subgraph of $G$ with edge set $E(N_1) \cup E(N_1, N_2)$ has at most $(2k + t/2 - 1)n$ edges, which follows form Lemma \ref{Many edges in a bipartition II}, along with $e(N_1) \le (2\kappa + t/2 -1) d(v)$ implies that
\begin{align*}
    2e(S_1) &\le (4\kappa + t - 2)n,\\
    e(L_1, S_1) &\le (2\kappa + t/2 - 1)n,\\
    e(N_1, S_2) &\le (2\kappa + t/2 - 1)n.
\end{align*}
Combining these inequalities with Lemma \ref{degrees of vertices in L}, we deduce that
\begin{equation}
    \begin{split}
    &2e(S_1)\alpha + 2e(L) + e(L_1, S_1)\alpha + e(N_1, S_2)\alpha\\
    \le &(4\kappa + t - 2)n\alpha + 2 \binom{|L|}{2} + (2\kappa + t/2 - 1)n\alpha + (2\kappa + t/2 - 1)n\alpha\\
    \le &10\kappa n\alpha,
    \end{split}
\end{equation}
for $n$ sufficiently large. Hence,  
\begin{equation}
\label{upperbound on lambdasq 2}
    \kappa nc \le d(v) c + \sum_{u \in S_1} \sum_{\substack{w \sim u \\ w \in L_1 \cup L_2}} \mathrm{x}_w + 10 \kappa n\alpha\le d(v) c + e(S_1, L_1 \cup L_2) + 10 \kappa n \alpha \le d(v) c + e(S_1, L_1 \cup L_2) + \frac{\epsilon^2 n}{2},
\end{equation}
where the last inequality is by \eqref{choice of constants}. 
Thus, using $d(v) < cn - \epsilon n$, we get that
\begin{equation}
    \left(\kappa - c + \epsilon\right)nc < \left(\kappa n - d(v)\right)c \le e(S_1, L_1 \cup L_2) + \frac{\epsilon^2 n}{2}.
\end{equation}

Since $v \in L'$, by \eqref{choice of constants} we have $c \geq \eta > \epsilon$. Using $\epsilon < c \le 1$, we obtain that
\begin{equation}
\label{degree bound using edges in some bipartite graph involving vertices of L}
    e(S_1, L_1 \cup L_2) > (\kappa -c)nc + \epsilon nc - \frac{\epsilon^2 n}{2} > (\kappa -1)nc + \frac{\epsilon^2 n}{2}.
\end{equation}

We show that $G$ contains a $K_{\kappa + 1, \kappa + t}$ which gives a contradiction using Lemma \ref{intersectiong circuits in large complete bipartite graphs}. We first prove the following claim.

\begin{claim}
\label{claim in many edges on (L, MUT)}
If $\delta:= \frac{\epsilon \alpha^2}{12500 \kappa^3}$, then there are at least $\delta n$ vertices inside $S_1$ with degree at least $\kappa$ in the bipartite subgraph $G[S_1, L_1 \cup L_2]$.
\end{claim}

\begin{proof}
Assume to the contrary that less than $\delta n$ vertices in $S_1$ have degree at least $\kappa$ in $G[S_1, L_1 \cup L_2]$. Then $e(S_1, L_1 \cup L_2) < (\kappa -1)|S_1| + |L|\delta n < (\kappa -1)(c-\epsilon)n +\epsilon n$, because $|S_1| \leq d(z)$ and by Lemma \ref{degrees of vertices in L}. Combining this with \eqref{degree bound using edges in some bipartite graph involving vertices of L} gives $ (\kappa -1)nc- (\kappa -2)n \epsilon > e(S_1, L_1\cup L_2) > (\kappa -1)nc + \frac{\epsilon^2 n}{2}$, a contradiction.
\end{proof}
Let $D$ be the set of vertices of $S_1$ that have degree at least $\kappa$ in $G[S_1, L_1 \cup L_2]$. Thus $|D| \ge \delta n$. Since there are only $\binom{|L|}{\kappa} \le \binom{12500 \kappa^3/\alpha^2}{\kappa}$ options for any vertex in $D$ to choose a set of $\kappa$ neighbors from, it implies that there exists some set of $\kappa$ vertices in $L_1 \cup L_2$ with at least $\delta n/ \binom{|L|}{\kappa} \ge  \frac{\epsilon \alpha^2 n}{12500 \kappa^3}/ \binom{12500 \kappa^3/\alpha^2}{\kappa}$ common neighbors in $D$. For $n$ large enough, this quantity is at least $\kappa + t$, thus $K_{\kappa +1, \kappa + t} \subset G[S_1, L_1 \cup L_2 \cup \{v\}]$, and we get our desired contradiction. 
\end{proof}

An important takeaway of Lemma \ref{mindegrees for vertices in L'} is that $d(z) \geq (1-\epsilon)n$ and for any $v\in L'$ we have $d(v) \geq (\eta - \epsilon)n$. By \eqref{choice of constants}, it follows that the neighborhoods of $z$ and $v$ intersect. So $L' \subset \{z\} \cup N_1(z) \cup N_2(z)$. 

Next we bound the number of edges in a bipartite graph contained in the closed 2-ball around $z$ : $N_2[z]$.

\begin{lemma}
\label{neighborhood of x, L}
For $z$ the vertex with $\mathbf{\mathrm{x}}_z =1$, we have $(1 - \epsilon)\kappa n \leq e(S_1, \{z\} \cup L_1\cup L_2) \leq (\kappa +\epsilon)n$.   
\end{lemma}
\begin{proof}

To obtain the lower bound, we use \eqref{upperbound on lambdasq 2}. Given $ \mathrm{x}_z= 1$, we have 
\[
\kappa n(1) \le d(z) + e(S_1, L_1 \cup L_2) + \frac{\epsilon^2 n}{2}.
\]
Since $e(S_1, \{z\} \cup L_1\cup L_2)= e(S_1, L_1\cup L_2) + d_{S_1}(z) = e(S_1, L_1\cup L_2) + d(z) - d_{L_1}(z) \ge \kappa n - |L_1| - \frac{\epsilon^2 n}{2}$, the lower bound follows as $|L_1| + \frac{\epsilon^2 n}{2} < \kappa \epsilon n$.

To obtain the upper bound, we assume toward a contradiction that for the vertex $z$, we have $e(S_1, \{z\} \cup L_1\cup L_2)> (\kappa + \epsilon)n$. We will obtain a contradiction via Lemma \ref{intersectiong circuits in large complete bipartite graphs} by showing that $K_{\kappa +1, \kappa + t} \subset G$. For this we prove the following claim.

\begin{claim}
\label{claim in many edges on (L, MUT) modified}
Let $\delta:= \frac{\epsilon \alpha^2}{12500 \kappa^3}$. With respect to the vertex $z$ there are at least $\delta n$ vertices inside $S_1$ with degree $\kappa$ or more in $G[S_1, L_1 \cup L_2]$.
\end{claim}

\begin{proof}
Assume to the contrary that less than $\delta n$ vertices in $S_1$ have degree at least $\kappa$ in $G[S_1, L_1 \cup L_2]$. Then $e(S_1, L_1 \cup L_2) < (\kappa-1)|S_1| + |L|\delta n < (\kappa-1)n +\epsilon n$, because $|S_1| \leq n$ and by Lemma \ref{degrees of vertices in L}. This contradicts our assumption that $e(S_1, \{z\} \cup L_1\cup L_2) > (\kappa + \epsilon)n$.
\end{proof}
Hence, there is a subset $D \subset S_1$ with at least $\delta n$ vertices such that every vertex in $D$ has degree at least $\kappa$ in $G[S_1, L_1\cup L_2]$. Since there are only at most $\binom{|L|}{\kappa} \le \binom{12500 \kappa^3 / \alpha^2}{\kappa}$ options for every vertex in $D$ to choose a set of $\kappa$ neighbors from, we have that there exists some set of $\kappa$ vertices in $L_1\cup L_2 \setminus \{z\}$ having a common neighborhood with at least $\delta n / \binom{|L|}{\kappa} \ge \delta n / \binom{12500 \kappa^3 / \alpha^2}{\kappa} = \frac{\epsilon \alpha^2}{12500 \kappa^3} n / \binom{12500 \kappa^3 / \alpha^2}{\kappa} \ge \kappa + t$ vertices. Thus, $K_{\kappa, \kappa + t} \subset G[S_1, L_1 \cup L_2]$ and $K_{\kappa +1, \kappa + t} \subset G[S_1, L_1 \cup L_2 \cup \{z\}]$, a contradiction by Lemma \ref{intersectiong circuits in large complete bipartite graphs}. Hence $e(S_1, \{z\} \cup L_1\cup L_2) \leq (\kappa +\epsilon)n$.
\end{proof}

Next we show that all the vertices in $L'$ in fact have Perron weight close to the maximum. 

\begin{lemma}
\label{precise size of L}
For all vertices $v\in L'$, we have 
$d(v) \ge \left(1-\frac{1}{8 \kappa^3}\right)n$ and $\mathbf{\mathrm{x}}_v \geq 1- \frac{1}{16 \kappa^3}$. Moreover, $|L'| = \kappa$.
\end{lemma}

\begin{proof}
Suppose we are able to show that $\mathbf{\mathrm{x}}_v \geq 1- \frac{1}{16\kappa^3}$, then using Lemma \ref{mindegrees for vertices in L'} and by \eqref{choice of constants} we have that
$d(v) \ge  \left(1-\frac{1}{8\kappa^3}\right)n$.
Then because every vertex $v \in L'$ has $d(v) \ge  \left(1-\frac{1}{8\kappa^3}\right)n$, we must have that $|L'| \le k$ else $G[S_1, L']$ contains a $K_{\kappa +1, \kappa +t}$, a contradiction by Lemma \ref{intersectiong circuits in large complete bipartite graphs}. 

Next, if $|L'| \le \kappa-1$ then using Lemma \ref{neighborhood of x, L} and \eqref{upperbound on lambdasq 2} applied to the vertex $z$ we have 
\[\kappa n \le \lambda^2 \le e(S_1, (L_1 \cup L_2) \cap L') + e(S_1, (L_1 \cup L_2)\setminus L')\eta + \frac{\epsilon n^2}{2} \le (\kappa -1)n + (\kappa+ \epsilon)n\eta + \frac{\epsilon^2 n}{2} < \kappa n\]
where the last inequality holds by \eqref{choice of constants} and gives the contradiction. Thus $|L'| = \kappa$. Thus, all we need to show is that $\mathbf{\mathrm{x}}_v \geq 1- \frac{1}{16 \kappa^3}$ for any $v \in L'$.

Now towards a contradiction assume that there is some vertex in $v \in L'$ such that $\mathbf{\mathrm{x}}_v < 1 - \frac{1}{16\kappa^3}$.
Then refining \eqref{upperbound on lambdasq 2} with respect to the vertex $z$ we have that 
\begin{equation}
\begin{split}
 \kappa n 
 & \le \lambda^2 < e(S_1(z), \{z\} \cup L_1(z) \cup L_2(z) \setminus \{v\}) + |N_1(z) \cap N_1(v)|\mathbf{\mathrm{x}}_v + \frac{\epsilon^2 n}{2}\\   
 & < (\kappa +\epsilon)n - |S_1(z) \cap N_1(v)| + |N_1(z) \cap N_1(v)|\left(1 - \frac{1}{16\kappa^3}\right) + \frac{\epsilon^2 n}{2}\\
 & = \kappa n + \epsilon n + |L_1(z) \cap N_1(v)| - |N_1(z) \cap N_1(v)|\frac{1}{16\kappa^3} + \frac{\epsilon^2 n}{2}.
\end{split}    
\end{equation}

Thus, using Lemma \ref{degrees of vertices in L} we have $\frac{|N_1(z) \cap N_1(v)|}{16\kappa^3} < \epsilon n + \frac{\epsilon^2 n}{2} + |L| \le 2 \epsilon n$.

But, $v \in L'$, thus $\mathbf{\mathrm{x}}_v \ge \eta$ and $d(v) \ge \left(\eta - \epsilon\right)n$, and so $|N_1(z) \cap N_1(v)| \ge  \left(\eta - 2\epsilon\right)n > 32\kappa^3 \epsilon n$ by \eqref{choice of constants}, a contradiction.
\end{proof}
Now we have $|L'| =  \kappa$ and every vertex in $L'$ has degree at least $\left(1 - \frac{1}{8 \kappa^3}\right)n$. Thus, the common neighborhood of vertices in $L'$ has at least $\left(1 - \frac{1}{8 \kappa^2}\right)n$ vertices. Let $R$ denote the set of vertices in this common neighborhood.  Let $\mathcal{E}$ be the set of remaining ``exceptional" vertices not in $L'$ or $R$. Thus $|\mathcal{E}| \le \frac{n}{8\kappa^2}$. We will now show that $\mathcal{E} = \emptyset$ and thus $G$ contains a large complete bipartite subgraph $K_{\kappa, n-\kappa}$. For this, we will first prove a bound on the sum of Perron weights in the neighborhood of any vertex.

\begin{lemma}
\label{minimum eigenweight in the neighborhood of a vertex}
For any vertex $v \in V(G)$, the Perron weight in the neighborhood of $v$ satisfies  $\displaystyle\sum_{w \sim v}\mathbf{\mathrm{x}}_w \geq \kappa - \frac{1}{16 \kappa^2}$. 
\end{lemma}

\begin{proof}
Clearly if $v \in L'$, we have \[\sum_{w \sim v}\mathbf{\mathrm{x}}_w = \lambda \mathbf{\mathrm{x}}_v \ge \lambda \left(1 - \frac{1}{16\kappa^3} \right) \geq \kappa - \frac{1}{16\kappa^2}.\] If $v \in R$, then \[\sum_{w \sim v}\mathbf{\mathrm{x}}_w \ge \sum_{\substack{w \sim v\\ w \in L'}}\mathbf{\mathrm{x}}_w \ge \kappa\left(1 - \frac{1}{16\kappa^3}\right) = \kappa - \frac{1}{16\kappa^2}.\]

Finally, let $v \in \mathcal{E}$. If $\sum_{w \sim v}\mathbf{\mathrm{x}}_w < \kappa - \frac{1}{16\kappa^2}$, consider the graph $H$ obtained from $G$ by deleting all edges adjacent to $v$ and adding the edges $uv$ for all $u \in L'$. Now since $\sum_{w \sim v}\mathbf{\mathrm{x}}_w < \kappa - \frac{1}{16\kappa^2}$ we have that $\rx^TA(H)\rx > \rx^TA(G)\rx$, and so by the Rayleigh principle $\lambda(H) > \lambda(G)$. However, there are no new intersecting even cycles $C_{2k_1, 2k_2, \ldots, 2k_t}$ that have isomorphic copies in $H$ but no isomorphic copies in $G$. To see this, assume to the contrary that there is an isomorphic copy of $C_{2k_1, 2k_2, \ldots, 2k_t}$ in $H$ but not in $G$. Then the $C_{2k_1, 2k_2, \ldots, 2k_t}$ has $2\kappa + t + 1$ vertices $v_1 (=v), v_2, \ldots, v_{2\kappa + t + 1}$ and $v$ has at most $\kappa$ neighbors in $C_{2k_1, 2k_2, \ldots, 2k_t}$ all of which lie in $L'$. However, the common neighborhood of $L'$, in $G$, has at least $\left(1 - \frac{1}{8\kappa^2}\right)n > 2\kappa + t + 1$ vertices. Therefore, $G$ must already contain an isomorphic of $C_{2k_1, 2k_2, \ldots, 2k_t}$, a contradiction.
\end{proof}

Let $K_{a, b}^{p},$ and $K_{a, b}^{m}$ denote the graphs obtained from the complete bipartite graph $K_{a, b}$ by adding into the part of size $b$ 
a path $P_3$ on $3$ vertices, and a matching with two edges $K_2 \cup K_2$, respectively. So, $K_{a, b}^{p} = a K_1 \vee ((b-3)K_1 \cup P_3)$ and $K_{a, b}^{m} = a K_1 \vee ((b-4)K_1 \cup 2K_2)$. Then we can make the following observation.
\begin{lemma}
\label{intersecting even cycles in almost bipartite graphs}
Let $2 \le k_1 \le k_2 \le \ldots \le k_t$ and $\kappa:= \sum_{i=1}^t k_i'$. If $k_t = 2$, then $K_{t, 2t + 1}^{p}$ contains the intersecting even cycle $C_{4, 4, \ldots, 4}$ consisting of $t$ intersecting $4$-cycles. If $k_t \ge 3$, then the graphs $K_{\kappa, \kappa + t + 1}^{p}$ and $K_{\kappa, \kappa + t + 1}^{m}$ both contain the intersecting even cycles $C_{2k_1, 2k_2, \ldots, 2k_t}$. 
\end{lemma}

It follows from Lemma \ref{intersecting even cycles in almost bipartite graphs} that if $k_t = 2$ then every vertex in $G[R]$ has degree less than 2. Further, if $k_t \ge 3$ then $e(R) \le 1$. Moreover, any vertex $v \in \mathcal{E}$ is adjacent to at most $\kappa + t - 1$ vertices in $R$, else $K_{\kappa +1, \kappa + t} \subset G[L' \cup \{v\}, R]$, a contradiction by Lemma \ref{intersectiong circuits in large complete bipartite graphs}. Finally, any vertex in $\mathcal{E}$ is adjacent to at most $\kappa-1$ vertices in $L'$ by the definition of $\mathcal{E}$.
We are now ready to show that $\mathcal{E}$ is empty and therefore $S = R$, so $G$ must contain the complete bipartite graph $K_{\kappa, n-\kappa}$.

\begin{lemma}
\label{E empty set}
The set $\mathcal{E}$ is empty and $G$ contains the complete bipartite graph $K_{\kappa, n- \kappa}$.
\end{lemma}

\begin{proof}
Assume to the contrary that $\mathcal{E} \neq \emptyset$. Recall that any vertex $r \in R$ satisfies $\mathbf{\mathrm{x}}_r < \eta$. Therefore, any vertex $v \in \mathcal{E}$ must satisfy
\[\sum_{u \sim v} \mathbf{\mathrm{x}}_u = \lambda \mathbf{\mathrm{x}}_v = \sum_{\substack{u \sim v \\ u \in L' \cup R}}\mathbf{\mathrm{x}}_u + \sum_{\substack{u \sim v \\ u \in \mathcal{E}}}\mathbf{\mathrm{x}}_u \le \kappa -1 + \left(\kappa + t -1\right)\eta +\sum_{\substack{u \sim v \\ u \in \mathcal{E}}}\mathbf{\mathrm{x}}_u.\]
Combining this with Lemma \ref{minimum eigenweight in the neighborhood of a vertex} gives 
\begin{equation}
\begin{split}
    \frac{\displaystyle\sum_{\substack{u \sim v \\ u \in \mathcal{E}}} \mathbf{\mathrm{x}}_u}{\lambda \mathbf{\mathrm{x}}_v} \geq \frac{\lambda \mathbf{\mathrm{x}}_v - (\kappa -1)-(\kappa + t - 1)\eta}{\lambda \mathbf{\mathrm{x}}_v} \ge 1 - \dfrac{(\kappa-1) + (\kappa + t - 1) \eta}{\kappa - \frac{1}{16 \kappa^2}} \ge \frac{4}{5\kappa},
\end{split}
\end{equation}
where the last inequality follows from \eqref{choice of constants}. Now consider the matrix $B = A(G[\mathcal{E}])$ and vector $y := \mathbf{\mathrm{x}}_{|_\mathcal{E}}$ (the restriction of the vector $\mathbf{\mathrm{x}}$ to the set $\mathcal{E}$). We see that for any vertex $v \in \mathcal{E}$

\[B \mathbf{\mathrm{y}}_v = \sum_{\substack{u \sim v \\ u \in \mathcal{E}}}  \mathbf{\mathrm{x}}_u \ge \frac{4}{5\kappa} \lambda \mathbf{\mathrm{x}}_v = \frac{4}{5\kappa} \lambda \mathbf{\mathrm{y}}_v.\]

Hence, by Lemma \ref{lower bound for spectral radius of nonnegative matrices}, we have that $\lambda(B) \geq \frac{4}{5\kappa}\lambda \ge \frac{4}{5}\sqrt{\frac{n}{\kappa}}$. Moreover, $\frac{n}{8\kappa^2} \ge |\mathcal{E}| > \lambda(B) \ge \frac{4}{5}\sqrt{\frac{n}{\kappa}}$, so $\mathcal{E}$ must have sufficiently many vertices to apply Lemma \ref{bounds on lambda} if $n$ is sufficiently large and $|\mathcal{E}| \neq 0$.  This is a contradiction to Lemma \ref{bounds on lambda} which gives $\lambda(B) \le \sqrt{5\kappa|\mathcal{E}|} \le \sqrt{5\kappa\frac{n}{8\kappa^2}} = \sqrt{\frac{5n}{8\kappa}}$, else $\mathcal{E}$ contains $C_{2k_1, 2k_2, \ldots, 2k_t}$. Therefore, $\mathcal{E}$ must be empty.
\end{proof}

\section{Proof of Theorems \ref{theorem intersecting even cycles} and \ref{theorem intersecting 4-cycles}}\label{section main proof}

It follows from Lemma \ref{E empty set} that $G$ contains the complete bipartite graph $K_{\kappa, n-\kappa}$, where the part on $\kappa$ vertices is the set $L'$ and the part on $n-\kappa$ vertices is the set $R$. By Lemma \ref{intersecting even cycles in almost bipartite graphs}, if $k_t = 2$, then $G[R] \subset M_{n-\kappa} = M_{n-t}$, so $G \subseteq F_{n, t}$, and if $k_t \ge 3$, then
$e(R) \leq 1$ in $G$. Hence $G \subset S_{n,\kappa}^+$. Since adding more edges will only increase the spectral radius, we see that the spectral radius is maximized if the vertices of $L'$ induce a clique $K_{\kappa}$ and the number of edges in $R$ is as large as possible. Thus $G \cong F_{n, t}$ (and $G \cong S_{n,\kappa}^+)$, for $k_t=2$ (and $k_t \ge 3$, respectively).
\qed

\section{Applications to related minor-free spectral Tur\' an problems}\label{section minor free}

By Theorem \ref{theorem intersecting 4-cycles} we have established that $F_{n, t}$ is the unique spectral extremal graph in $\rSPEX(n, C_{4, 4, \ldots, 4})$, for $n$ sufficiently large. This answers Problem 7.1 of \cite{he2023spectral}. Furthermore, in \cite{he2023spectral}, He, Li and Feng studied the spectral Tur\' an problem for $H$-minor-free graphs. 
For two graphs $H$ and $G$, we say that $H$ is a minor of $G$ if we can obtain a graph isomorphic to $H$ from $G$ by performing a sequence of operations consisting of vertex deletions, edge deletions or edge contractions to $G$. If no such sequence of operations can yield an isomorphic copy of $H$, we say that $G$ is $H$-minor-free.
For some choice of $k_1, k_2,\ldots, k_t$, the minor-free spectral Tur\' an problem for $C_{k_1, k_2, \ldots, k_t}$ involves determining the $n$-vertex, $C_{k_1, k_2, \ldots, k_t}$-minor-free graphs that achieve maximum spectral radius. For ease of communication, we will refer to such a graph as a \textit{minor-free spectral extremal graph}.  
In \cite{he2023spectral}, for $n$ sufficiently large, the authors determined that for $C_{k_1, k_2, \ldots, k_t} = C_{3, 3, \ldots, 3}$ (or $C_{4, 4, \ldots, 4}$), the graph $S_{n, t}$ (or $F_{n, t}$, respectively) is the unique $C_{k_1, k_2, \ldots, k_t}$-minor-free spectral extremal graph on $n$ vertices. 

In Theorem \ref{theorem intersecting even cycles}, we have established that $S_{n, \kappa}^+$ is the unique spectral extremal graph in $\rSPEX(n, C_{2k_1, 2k_2, \ldots, 2k_t})$ for any choice of $k_i, 1 \le i \le t$, where $k_t \ge 3$, whenever $n$ is taken to be sufficiently large. We state the following two remarks which when combined with Theorem \ref{theorem intersecting even cycles} give us a general minor-free spectral Tur\' an result for intersecting even cycles.

\begin{remark}
\label{remark spectral extremal larger than minor-free spectral extremal}
   If $G$ is an $H$-minor-free graph, then $G$ does not contain any subgraph isomorphic to $H$. Therefore, for $n$ large enough, if $G$ is a $C_{2k_1, 2k_2, \ldots, 2k_t}$-minor-free spectral extremal graph on $n$ vertices and $k_t \ge 3$, then $\lambda(G) \le \lambda(S_{n, \kappa}^+)$. Similarly, if $k_t = 2$ then $\lambda(G) \le \lambda(F_{n, t})$.
\end{remark}
\begin{remark}
\label{Sn,k+ is also minor-free}
   For $k_t \ge 3$, the graph $S_{n, \kappa}^+$ is $C_{2k_1, 2k_2, \ldots, 2k_t}$-minor free. This is true because $S_{n, \kappa}^+$ is $C_{2k_1, 2k_2, \ldots, 2k_t}$-free and any sequence of vertex deletions, edge deletions and edge contractions to $S_{n, \kappa}^+$ only yield a subgraph of $S_{n, \kappa}^+$. Similarly, $F_{n, t}$ is $C_{4, 4, \ldots, 4}$-minor-free, where $C_{4, 4, \ldots, 4}$ consists of $t$ intersecting $4$-cycles.
\end{remark}

We are now ready to state our theorem for minor-free spectral Tur\' an problems on intersecting even cycles. Note that the result for $C_{4, 4, \ldots, 4}$ already follows from \cite{he2023spectral} and is combined into the following theorem for completion.

\begin{theorem}
\label{theorem minor free}
    For integers $2 \le k_1 \le k_2 \le \ldots k_t$, let $G$ be a $C_{2k_1, 2k_2, \ldots, 2k_t}$-minor-free spectral extremal graph on $n$ vertices. If $n$ is sufficiently large, we have $\rSPEX(n, C_{2k_1, 2k_2, \ldots, 2k_t}) =  \{G\}$. In particular,
    \[G = F_{n, t} \hspace{0.25cm} \text{ for } k_t = 2, \text{ and }\]
    \[G = S_{n, \kappa}^+ \hspace{0.25cm} \text{ for } k_t \ge 3.\]

\end{theorem}

We end by drawing attention to a question that appears in \cite{he2023spectral}.

\begin{question}
    Determine the graphs $H$ for which $\rSPEX(n, H) = \textrm{m-SPEX}(n,H)$, where $\textrm{m-SPEX}(n, H)$ denotes the set of $H$-minor-free spectral extremal graphs on $n$ vertices.  
\end{question} 

Theorem \ref{theorem minor free} gives us that for $H = C_{2k_1, 2k_2, \ldots, 2k_t}$, we have $\rSPEX(n, H) = \textrm{m-SPEX}(n,H)$.
Indeed, it is also true that if $G \cong K_a \vee \left(K_{b_1} \cup K_{b_2} \cup \ldots \cup K_{b_s}\right)$ and $G \in \rSPEX(n, H)$ for some graph $H$ (so $n = a + \left(\sum_{i=1}^s b_i \right)$), then $G$ must also be an $H$-minor-free spectral extremal graph. 
To see this, observe that we have assumed that $G$ has no subgraph isomorphic to $H$. Further, due to the structure of $G$, any minor of $G$ is isomorphic to a subgraph of $G$. Thus $G$ is $H$-minor-free. Since any $H$-minor-free graph is also $H$-free, this implies that $G$ must be an $H$-minor-free spectral extremal graph and any other minor-free spectral extremal graph of order $n$ is also a spectral extremal graph.

\vspace{0.5cm}
\textbf{Acknowledgements:} The author is grateful to Prof.\ Sebastian Cioab\u a and Prof.\ Michael Tait for their useful advice during the preparation of this paper. We also thank Dr.\ Yongtao Li for bringing to attention \cite{he2023spectral} and for productive discussions on the same. Also, we show our appreciation to the referees for numerous careful comments and corrections that have improved this exposition. 

\bibliographystyle{plain}
	\bibliography{bib.bib}
\end{document}